\documentclass[a4paper]{article}
\usepackage{amsmath,amssymb,amsfonts,amsthm} 
\usepackage[utf8]{inputenc}
\usepackage{tikz,relsize} 

\usepackage{pgf}
\usepackage{tikz-cd}

\usepackage{hyperref}
\usepackage[noblocks]{authblk}

\usepackage[T1]{fontenc}

\usepackage{marginnote}
\usepackage{color}
\usepackage{changepage}
\strictpagecheck
\usepackage[normalem]{ulem}
\makeatletter
\def\redsout{%
\ULdepth=.5ex%
\ifmmode \ULdepth=1.0ex \fi%
\bgroup\markoverwith{\textcolor{red}{\rule[\ULdepth]{2pt}{1.5pt}}}\ULon%
}%
\makeatother
\usepackage[shortlabels]{enumitem}
\setlist[enumerate, 1]{label={\arabic*.}}
\newlist{enumroman}{enumerate}{1}
\setlist[enumroman, 1]{label={\arabic*)}}

\usepackage[backend=bibtex]{biblatex}
\bibliography{bibl}

\newcommand{\ses}[7]{\ensuremath{#1 \longrightarrow #2 \stackrel{#3}{\longrightarrow} #4 \stackrel{#5}{\longrightarrow} #6 \longrightarrow #7}}

\newcommand{\ns}[1]{\mathbb{#1}}
\newcommand\N{\ns{N}}
\newcommand\Z{\ns{Z}}

\newcommand\R{\ns{R}}

\usepackage[mathscr]{eucal}

\newcommand{\D}{\mathcal{D}}

\renewcommand{\S}{\mathcal{V}}

\newcommand{\s}[1]{\mathit{#1}}
\newcommand{\conj}[1]{\overline{#1}}

\DeclareMathOperator{\sw}{sw}
\DeclareMathOperator{\PA}{\mathcal P}
\DeclareMathOperator{\smr}{smr}
\DeclareMathOperator{\smc}{smc}
\DeclareMathOperator{\HWM}{\mathcal{H}}
\DeclareMathOperator{\Map}{Map}

\DeclareMathOperator{\Isom}{Isom}

\DeclareMathOperator{\GL}{GL}
\DeclareMathOperator{\GO}{O}
\DeclareMathOperator{\SO}{SO}
\DeclareMathOperator{\Hom}{Hom}
\DeclareMathOperator{\img}{Im}
\DeclareMathOperator{\vspan}{span}
\DeclareMathOperator{\Spin}{Spin}
\DeclareMathOperator{\GSpinc}{Spin^c}
\usepackage{xspace}
\newcommand{\spinc}{\text{spin$^c$}\xspace}
\newcommand{\Spinc}{\text{Spin$^c$}\xspace}

\numberwithin{equation}{section}
\newtheorem{thm}{Theorem}[section]
\newtheorem{lm}[thm]{Lemma}
\newtheorem{cor}[thm]{Corollary}
\newtheorem{prop}[thm]{Proposition}

\theoremstyle{definition}
\newtheorem{df}[thm]{Definition}
\newtheorem{ex}[thm]{Example}

\theoremstyle{remark}
\newtheorem{rem}[thm]{Remark}

\newenvironment{emphasis}%
{%
\setlength{\parindent}{0pt}%
\vspace{\topsep}\em%
}%
{\vspace{\topsep}}%

\author[*]{R. Lutowski} 
\author[*]{J. Popko}
\author[*]{A. Szczepa\'{n}ski}
\affil[*]{Institute of Mathematics, University of Gda\'{n}sk, Gda\'{n}sk, Poland}

\title{\Spinc structures on Hantzsche-Wendt manifolds}

\begin{document}

\maketitle

\begin{abstract}
Using a combinatorial description of Stiefel-Whitney classes of closed flat manifolds with diagonal holonomy representation, we show that 
no Hantzsche-Wendt manifold of dimension greater than three does not admit a \spinc structure.
\end{abstract}

\section{Introduction} 

Hantzsche-Wendt manifolds are examples of flat manifolds, i.e. closed Riemannian manifolds with vanishing sectional curvature. They are generalizations of the three-dimensional flat orientable manifold defined in \cite{HW36} and, following \cite{RS05}, we say that:

\begin{emphasis}
An orientable $n$-dimensional flat manifold is Hantzsche-Wendt if and only if its holonomy group is an elementary abelian $2$-group of rank $n-1$.
\end{emphasis}

Every $n$-dimensional flat manifold $X$ occurs as a quotient space of the action of $\Gamma$ on the euclidean space $\R^n$, where $\Gamma$ is a Bieberbach group, i.e. a torsion-free, co-compact and discrete subgroup of the group $\Isom(\R^n) = \GO(n) \ltimes \R^n$ of isometries of $\R^n$. $X$ is an Eilenberg-MacLane space of type $K(\Gamma,1)$. 
By Bieberbach theorems (see \cite{Sz12}), $\Gamma$ is defined by the following short exact sequence
\begin{equation}
\label{eq:bieberbach}
\ses{0}{\Z^n}{\iota}{\Gamma}{\pi}{G}{1},
\end{equation}
where $\iota(\Z^n)$ is the maximal abelian normal subgroup of $\Gamma$, $G$ is finite and coincides with the holonomy group of $X$. Moreover, by conjugations in $\Gamma$, $G$ acts in a natural way on $\Z^n$, giving it the structure of a $G$-module. 

Taking into account the above definition we will say that a Bieberbach group $\Gamma \subset \Isom^+(\R^n) = \SO(n) \ltimes \R^n$, defined by \eqref{eq:bieberbach}, is a \emph{Hantzsche-Wendt group} and $X=\R^n/\Gamma$ is a \emph{Hantzsche-Wendt manifold} (\emph{HW-group} and \emph{HW-manifold} for short) if $G \simeq C_2^{n-1}$.

Among many properties of HW-manifolds which were objects of research one can list the following: they exist only in odd dimensions \cite{MR99}, they are rational homology spheres \cite{Sz83} and cohomologically rigid \cite{PS16}. If $\Gamma$ is a HW-group then it is an epimorphic image of a certain Fibonacci group \cite{LM18} and if its dimension is greater than or equal to $5$, then its commutator and translation subgroups coincide \cite{P07}. One of the crucial -- for the purposes of this paper -- properties of HW-manifolds (HW-groups) is the one described in \cite{RS05}: they are \emph{diagonal}, i.e. there exists a $\Z$-basis $\mathcal{B}$ of the $G$-module $\Z^n$ such that 
\[
gb = \pm b
\]
for every $b \in \mathcal{B}$ and $g \in G$.

Now, let $n \geq 3$. The fundamental group $\pi_1(\SO(n))$ of the special orthogonal group $\SO(n)$ is of order $2$. The spin group $\Spin(n)$ is its double cover -- and the universal cover in fact. Let $\lambda_n \colon \Spin(n) \to \SO(n)$ be the covering map. A spin structure on a smooth orientable manifold $X$ is an equivariant lift of its frame bundle via $\lambda_n$. Its existence is equivalent to the vanishing of the second Stiefel-Whitney class $w_2(X)$ of $X$, see \cite[page 40]{F00}. In the case when $X$ is flat, it is closely connected to the Sylow 2-subgroup of its holonomy group \cite{DSS06} and can be determined by an algorithm \cite{LP15}. The three-dimensional HW-manifold has a spin structure (see 
\cite[Theorem VII.1]{Ki89}). But this is the only case -- by \cite[Example 4.6]{MP06} no other HW-manifold admits any spin structure. 

In the case when there are no spin structures, one can consider their complex analogue. We have that
\[
\GSpinc(n) := \left( \Spin(n) \times S^1 \right)/\langle (-1,-1) \rangle = \Spin(n) \times_{C_2} S^1
\]
is the double cover of $\SO(n) \times S^1$ for which  the \spinc structure is defined -- in analogy to the spin case -- with the covering map $\bar\lambda_n \colon \GSpinc(n) \to \SO(n) \times S^1$ given by
\[
\bar\lambda_n[x,z] := \left(\lambda_n(x),z^2\right).
\]
The manifold $X$ has a \spinc structure if and only if $w_2(X)$ is the $\bmod\,2$ reduction of some integral cohomology class $z \in H^2(X,\Z)$, see \cite[page 49]{F00}. We immediately get that existence of spin structures determines existence of \spinc structures -- in fact the former induces the latter, but not the other way around. For example, by an unpublished work \cite{TV} all orientable $4$-manifolds have some \spinc structures, but by \cite{PS10}, $3$ of the $27$ flat ones don't have any.

In this paper we prove that every HW-manifold of dimension greater than or equal to $5$ does not admit any \spinc structure. Note that some examples of non-\spinc HW-manifolds were given in \cite{GS13}.

The tools that we use have been introduced in \cite{PS16} and used for example in \cite{LPPS19}. They proved their effectiveness in cohomology-related properties of diagonal manifolds.

The structure of the paper is as follows. Sections 2 and 3 give a quick glance on a way of the encoding diagonal manifolds and their Stiefel-Whitney classes by certain matrices. This has been already presented in more detail in \cite{PS16} and \cite{LPPS19}. In Section 4 we give one of two theorems on conditions equivalent to the existence of \spinc structures on HW-manifolds. For our further analysis we introduce HW-matrices. This description of HW-manifolds was introduced in \cite{PS16} and is in fact one-to-one with the one given in \cite{MR99}. Technical Section 6 gives us some properties and formulas for matrices that we work with. The second theorem on conditions equivalent to the existence of \spinc structures on HW-manifolds is given in Section 7. After that we give a very specific form to a matrix which describes a (possible) \spinc HW-manifold and at last we show that this form can never occur. This proves that no HW-manifold can admit a \spinc structure.

\section{Diagonal flat manifolds}

In this section we give a combinatorial description of diagonal flat manifolds. This language is essential in the analysis of the Steifel-Whitney classes of such manifolds.

\begin{rem}
For any matrix $A$ by $A_{ij}, A_{i,j}$ or $A_i[j]$ we shall denote the element in the $i$-th row and $j$-th column of $A$. By $A_i$ we shall understand the $i$-th row of $A$.
\end{rem}

\begin{rem}
Let $k \in \N$. Cyclic groups of order $k$ with multiplicative and additive structure will be denoted by $C_k$ and $\Z_k:=\Z/k$, respectively. Note that in the natural way $\Z_k$  is ring and possibly -- a field.
\end{rem}

Suppose $\Gamma$ is a Bieberbach group defined by the short exact sequence \eqref{eq:bieberbach}. As mentioned in the introduction, conjugations in $\Gamma$ define a $G$-module $\Z^n$. To be a bit more precise, corresponding representation $\rho\colon G \to \GL_n(\Z)$ is called an \emph{integral holonomy representation of $\Gamma$} and it is given by the formula
\[
\rho_g(z)= \iota^{-1}(\gamma  \iota(z) \gamma^{-1}),
\]
where $z\in \Z^n, g \in G$ and $\gamma \in \Gamma$ is such that $\pi(\gamma) = g$. 
We will call $\Gamma$ \emph{diagonal} or \emph{of diagonal type}  if the image of $\rho$ is a subgroup of the group
\[
D = \{A \in\GL(n,\Z) : A_{ij} = A_{ji} = 0 \text{ and } A_{ii} = \pm 1 \text{ for } 1\leq i < j \leq n\} \cong C_2^n
\]
of diagonal matrices of $\GL(n,\Z)$.
It follows that $G = C_2^k$ for some $1\leq k\leq n-1$.

Let $S^1=\R/\Z$. As in  \cite{PS16} and \cite{LPPS19},
we consider the automorphisms $g_i:S^1\to S^1,$ given by
\begin{equation}\label{dictionary1}
g_0([t]) = [t], \;\; g_1([t]) =
\bigg[t+\frac{1}{2}\bigg], \;\; g_2([t]) = [-t], \;\;  g_3([t]) =\bigg[-t+\frac{1}{2}\bigg],
\end{equation}
for $t \in \R$.
Let $\D = \{ g_i\mid i = 0,1,2,3\}$. It is easy to see that $\D \cong C_2\times C_2$ and $g_3 = g_1 g_2.$ We define an action of $\D^n$ on $T^n$ by
\begin{equation}\label{action}
(t_1,\dots,t_n)(z_1,\dots, z_n) = (t_1 z_1,\dots,t_n z_n),
\end{equation}
for $(t_1,\dots,t_n)\in {\D}^n$ and $(z_1,\dots,z_n)\in T^n = \underbrace{S^1\times\dots\times S^1}_n.$

\noindent Any minimal set of generators of a group $C_2^d\subseteq {\D}^n$ defines a $({d\times n} )$-matrix with entries in $ {\D}$ which in turn defines a matrix $A$ with entries in the set $\S=\{\s0, \s1, \s2, \s3\}$ under the identification $\s i \leftrightarrow g_i$,  $0\leq i \leq 3$. Note that elements of $\S$ are written in italic. 

\begin{df}
The structure of an additive group on $\S$ is given by
\[
i+j=k \Leftrightarrow g_ig_j=g_k,
\]
for $i,j,k \in \S$. This way $\S=\Z_2 \oplus \Z_2$ is in the natural way a $\Z_2$-vector space.
\end{df}

\begin{ex}
The three-dimensional HW-group has generators:
\[
\left(
\begin{bmatrix}
1 & 0& 0\\
0 &-1& 0\\
0 & 0&-1\\
\end{bmatrix},
\begin{bmatrix}
\frac{1}{2}\\
\frac{1}{2}\\
0\\
\end{bmatrix}
\right),
\left(
\begin{bmatrix}
-1& 0& 0\\
 0& 1& 0\\
 0& 0&-1\\
\end{bmatrix},
\begin{bmatrix}
0\\
\frac{1}{2}\\
\frac{1}{2}\\
\end{bmatrix}
\right),
\]
hence the corresponding matrix $A \in \S^{2 \times 3}$ is of the form
\[
A = \begin{bmatrix}
\s1 & \s3 & \s2\\
\s2 & \s1 & \s3\\
\end{bmatrix}.
\]
\end{ex}

\begin{rem}
\label{rem:injection}
Whenever our calculations involve $\Z_2=\{0,1\}$ and $\S$, it is done by identifying $\Z_2$ with the subgroup $\{\s0,\s1\} < \S$.
\end{rem}

We have the following characterization of the action of $C_2^d$ on $T^n$ and the associated  orbit space $T^n/C_2^d$ via the matrix $A$.  

\begin{lm}[{\cite[page 1050]{PS16}}]
\label{char}  Let $C_2^d\subseteq {\D}^n$ and define the matrix $A\in \S^{d\times n}$ as above. Then:
\begin{enumerate}[(i)]
\item the action of $C_2^d$ on $T^n$ is free if and only if there is  $\s1$ in the sum of any distinct collection of rows of $A$,
\item  $C_2^d$ is the holonomy group of $T^n/C_2^d$ if and only if there is either $\s2$ or $\s3$ in the sum of any distinct collection of rows of $A$.
\end{enumerate}
\end{lm}
\noindent When the action of $C_2^d$ on $T^n$ defined by (\ref{action}) is free, we will say that the associated  matrix $A$ is \emph{free} and we will call it the \emph{defining matrix} of $T^n/C_2^d$. In addition, when $C_2^d$ is the holonomy group of $T^n/C_2^d$, we will say  that $A$ is \emph{effective}.

\section{Stiefel-Whitney classes of diagonal flat manifolds}

The goal of this section is to introduce a notation and some basic results on Stiefel-Whitney classes of diagonal flat manifolds. For more precise description see \cite{LPPS19} and \cite{PS16}.

Let $n \in \N$ and $\Gamma$ be an $n$-dimensional diagonal Bieberbach group, given by the extension \eqref{eq:bieberbach}, with non-trivial holonomy group
$G=C_2^d$ ($d > 0$). Let $A \in \S^{d \times n}$ be a defining matrix of the corresponding flat manifold $X=\R^n/\Gamma = T^n/C_2^{d}$.

It is well-known that
\[
H^*(C_2^{d};\Z_2)\cong \Z_{2}[x_1,\dots,x_{d}],
\] 
where $\{x_1,\dots, x_{d}\}$ is a basis of $H^{1}(C_2^{d},\Z_2) = \Hom(C_2^{d},\Z_2)$ (see \cite[Theorem 1.2]{CMR10}). 
Let 
\[
\pi^*\colon H^*(C_2^{d},\Z_2) \to H^*(\Gamma, \Z_2)
\]
be the induced cohomology ring homomorphism. By \cite[Proposition 3.2]{LPPS19} the total Stiefel-Whitney class is given by
\[
w(X) = \pi^*( \sw ) \in H^*(\Gamma,\Z_2),
\]
where
\begin{equation}
\label{eq:sw}
\sw = \prod_{j=1}^{n} (1+\alpha_j + \beta_j).
\end{equation}
In the above formula for every $1 \leq j \leq n$, $\alpha_j,\beta_j \in H^1(C_2^d,\Z_2)$ are the cocycles defined by
\[
\alpha_j = \sum_{k=1}^d \alpha(A_{kj})x_k, \beta_j = \sum_{k=1}^d \beta(A_{kj})x_k
\]
and the linear homomorphisms $\alpha,\beta \in \Hom_{\Z_2}(\S,\Z_2)$ are uniquely defined by the following rules
\[
\alpha(\s2)=\beta(\s3)=1 \text{ and } \alpha(\s3)=\beta(\s2)=0.
\]

Let 
\[
\pi^*_{(i)} \colon H^i(C_2^{n-1},\Z_2) \to H^i(\Gamma, \Z_2)
\]   
be the induced group cohomology homomorphism (restriction of $\pi^*$ to the $i$-th gradation), for $0 \leq i \leq n$. Using again \cite[Proposition 3.2]{LPPS19} and the five-term exact sequence for the extension \eqref{eq:bieberbach} (see \cite[Formula (7)]{LPPS19}) we get 
\begin{lm}
\label{lm:five-term-corollary}
$\pi^*_{(1)}$ is injective and the kernel of $\pi^*_{(2)}$ is spanned by
\[
\theta_j = \alpha_j \cup \beta_j = \alpha_j \beta_j
\]
for $1 \leq j \leq n$.
\end{lm}

\begin{rem}
Note that the polynomials $\sw, \alpha_j, \beta_j, \theta_j$, where $1 \leq j \leq n$, can be defined for any matrix $A \in \S^{d \times n}$. To emphasize this connection or in the case when it won't be clear from the context, we will add the superscript $A$ to them and write $\sw^A$ for example.
\end{rem}

\section{Bockstein maps and spin$^c$ structures}

We will keep the notation of the previous section and restrict our attention to the case of Hantzsche-Wendt manifolds of dimension greater than or equal to 5. Hence $n \geq 5$ is an odd integer and $d=n-1$. Let $\beta_\Gamma$ and $\tilde{\beta}_\Gamma$ be the Bockstein homomorphisms of cohomology groups of $\Gamma$ associated to the short exact sequences
\begin{equation}
\label{eq:bockstein}
\ses{0}{\Z_2}{\cdot 2}{\Z_4}{\text{mod}_2}{\Z_2}{0}
\end{equation}
and
\begin{equation}
\label{eq:bockstein2}
\ses{0}{\Z}{\cdot 2}{\Z}{\text{mod}_2}{\Z_2}{0}
\end{equation}
respectively. If $\rho \colon H^2(\Gamma,\Z) \to H^2(\Gamma,\Z_2)$ is the homomorphism induced by the $\bmod\,2$ map, then we have the following commutative diagram
\[
\begin{tikzcd}
H^1(\Gamma,\Z) \arrow{r} & H^1(\Gamma,\Z_2)  \arrow{r}{\tilde{\beta}_\Gamma} \arrow{dr}{\beta_\Gamma} & H^2(\Gamma,\Z) \arrow{d}{\rho} \\
 & & H^2(\Gamma,\Z_2)
\end{tikzcd}
\]
with the row forming an exact sequence (see \cite[Chapter 3.E]{H02}). By \cite[Theorem 3.1]{P07} $H_1(\Gamma) \cong \Z_2^{n-1}$. By \cite[Theorem 9.2]{Sz12} $H_2(\Gamma)$ is a finite group. Moreover from the universal coefficient theorem (\cite[Theorem 3.2]{H02}),
\[
H^1(\Gamma,\Z) = 0 \text{ and } H^1(\Gamma,\Z_2) \cong H^2(\Gamma,\Z) \cong \Z_2^{n-1}.
\]
Hence $\tilde{\beta}_\Gamma$ is an isomorphism and $\img \beta_\Gamma = \img \rho$. 

Let $\beta$ be the Bockstein homomorphism of cohomology groups of $C_2^{n-1}$ associated to the extension \eqref{eq:bockstein}. The homomorphism $\pi$ induces the commutative diagram
\[
\begin{tikzcd}
H^1(C_2^{n-1},\Z_2) \arrow{r}{\beta}  \arrow{d}{\pi^*_{(1)}}& H^2(C_2^{n-1},\Z_2) \arrow{d}{\pi^*_{(2)}} \\
H^1(\Gamma,\Z_2) \arrow{r}{\beta_\Gamma} & H^2(\Gamma,\Z_2)
\end{tikzcd}
\]
By Lemma \ref{lm:five-term-corollary}, $\pi^*_{(1)}$ is a monomorphism of the elementary abelian $2$-groups of rank $n-1$, hence it is an isomorphism and
\[
\img \rho = \img \beta_\Gamma = \img \beta_\Gamma \pi^*_{(1)} = \img \pi^*_{(2)}\beta = \img \pi^*\beta.
\] 
Let $\sw_2$ be the sum of degree $2$ terms of the polynomial $\sw$. Then $w_2(X) = \pi^*(\sw_2)$ and by definition the manifold $X=\R^n/\Gamma$ admits a \spinc structure if and only if $\pi^*(\sw_2) \in \img \pi^*\beta$. This condition is obviously equivalent to
\[
(\sw_2 + \ker \pi^*) \cap \img \beta \neq \emptyset.
\]
In addition, one can easily show that for every $x \in H^1(C_2^{n-1},\Z_2)$ and $a,b \in C_2^{n-1}$ we have
\[
\beta(x)(a,b) = x(a)x(b) = x^2(a,b),
\]
hence $\beta(x)=x^2$ and $\pi^*(\beta(x)) = \pi^*(x)^2$. Similarly, $\beta_\Gamma(f) = f^2$ for $f \in H^1(\Gamma,\Z_2)$.

Using Lemma \ref{lm:five-term-corollary}, we get the following theorem:
\begin{thm}
\label{thm:spinc_condition1}
Assume that $n\geq 5$ is an odd integer and $X$ is an $n$-dimensional Hantzsche-Wendt manifold. Let $A \in \S^{n-1 \times n}$ be a defining matrix of $X$. Then the following conditions are equivalent:
\begin{enumerate}
\item $X$ admits a \spinc structure.
\item $w_2(X) \in H^*(\Gamma,\Z_2)$ is a square.
\item There exists $x \in H^1(\Z_2^{n-1},\Z_2)$ such that $x^2 + \sw^A_2 \in \vspan\{\theta^A_1,\ldots,\theta^A_n\}.$
\end{enumerate}
\end{thm}

\section{HW matrices}


Let $n \in \N$. Every $n$-dimensional HW-manifold $X$ defines some matrix $A \in \S^{n-1 \times n}$. For the purpose of investigating \spinc properties of $X$ it will be more convenient to work with a square matrix -- a HW-matrix. HW-matrices were defined in \cite{PS16}.

Let $Z$ be a finite set. By $\PA(Z)$ we denote the algebra (over the field $\Z_2$) of subsets of $Z$. Just recall that the addition and multiplication in $\PA(Z)$ are defined by the symmetric difference and intersection respectively:
\[
\forall_{A,B \in \PA(Z)} A+B := (A \setminus B) \cup (B \setminus A) \text{ and } A \cdot B := A \cap B.
\]
Empty set and $Z$ are zero and one of this algebra, respectively. Let us note without a proof:
\begin{lm} \ 
\label{lm:set_algebra}
\begin{enumerate}
\item The map $|\cdot|_2 \colon \PA(Z) \to \Z_2$, given by 
\[U \mapsto |U| \bmod 2,\]
is linear. 
\item Every permutation of $Z$ is an algebra automorphism of $\PA(Z)$. 
\end{enumerate}
\end{lm}
\begin{rem}
We will use the notation $\PA_d := \PA(\{1,\ldots,d\})$ for $d \in \N$.
\end{rem}

\begin{df}
Let $d,n \in \N$ and $A \in \S^{d \times n}$. For $S \in \PA_n$ and $1 \leq i \leq d$ we have the sum of elements of the $i$-th row $A$ which lie in the columns from the set $S$:
\[
\smr_i^S(A) := \sum_{j \in S} A_{ij}
\]
and we denote $\smr_i^{\{1,\ldots,n\}}(A)$ simply by $\smr_{i}(A)$. In a similar way we define the column sums $\smc_j^S(A)$ (and $\smc_j(A)$) for $S \in \PA_d$ and $1 \leq j \leq n$.
Moreover, we define a map $J_A \colon \PA_d \to \PA_n$ as follows
\[
J_A(U) := \left\{ j : \smc_j^U(A) = \s1 \right\}.
\]
\end{df}

\begin{df}
The exists the unique $\Z_2$-linear involution $\conj{\cdot} \colon \S \to \S$ which maps $\s2$ to $\s3$. We call this map a \emph{conjugation}. To be explicit, we have
\[
\conj{\s0} = \s0, \conj{\s1} = \s1, \conj{\s2} = \s3 \text{ and } \conj{\s3} = \s2.
\]
\end{df}

\begin{df}
Let $A$ be a matrix with coefficients in $\S$. We call $A$:
\begin{itemize}
\item \emph{self-conjugate} if $A^t = \overline{A}$, where $A^t$ is the transpose of $A$ and $\overline{A}$ is the element-wise conjugate of $A$;
\item \emph{distinguished} if it has $\s1$ on the main diagonal and $\s2$ or $\s3$ everywhere else.
\end{itemize}
\end{df}

\begin{rem}
Recall that we speak about a \emph{principal submatrix} of a given matrix if the sets of row and column indices which define it are the same (see \cite[Definition 6.2.5]{BS89} for example). We immediately get, that principal submatrices of self-conjugate and distinguished matrices are themselves self-conjugate and distinguished, respectively.
\end{rem}

\begin{lm}
\label{lm:column_sums}
Let $A \in \S^{k \times n}$ be distinguished, where $k \leq n$. Then the possible values for $\smc_j(A)$, where $1 \leq j \leq n$ are given by the following table:
\[
\begin{array}{r|c|c}
          & j \leq k & j > k \\ \hline
2 \mid k  & \s2\text{ or }\s3 & \s0\text{ or }\s1\\
2 \nmid k & \s0\text{ or }\s1 & \s2\text{ or }\s3
\end{array}
\]
\end{lm}
\begin{proof}
Simple calculation of the parity of the number of $\s2$ and $\s3$ in each column.
\end{proof}

\begin{df}[{\cite[Definition 2]{PS16}}]
Let $n \in \N$. We will call $A \in \S^{n \times n}$ a \emph{HW-matrix} if:
\begin{enumroman}
\item $A$ is distinguished;
\item $\smc_j(A) = \s0$ for every $1 \leq j \leq n$;
\item $J_A(U) \neq 0$ for every $U \in \PA_n \setminus \{0,1\}$.
\end{enumroman}
The set of HW-matrices of degree $n$, or $n$-HW-matrices for short, will be denoted by $\HWM_n$.
\end{df}

By Lemma \ref{lm:column_sums} we immediately get:
\begin{cor}
Every HW-matrix is of odd degree.
\end{cor}

\begin{rem}
\label{rem:hwmats}
We can think of the above definition as coming from the encoding Hantzsche-Wendt groups presented in \cite{MR99}. In connection to this description we note:
\begin{enumerate}
\item Any row of a HW-matrix may be removed and the corresponding torus quotient will remain the same. In other words, the removal will make the matrix a defining and effective one for the same HW-manifold.
\item Every HW-manifold defines some HW-matrix.
\item There is an action of the group $G_n := C_2 \wr S_n$ on the set $\S^{n \times n}$. 
Namely, for every $A \in \S^{n \times n}$ we have that
\begin{enumerate}
\item $c_k$ conjugates the $k$-th column of $A$, where $c_k \in C_2^n$ has non-trivial element of $C_2$ in the $k$-th coordinate only;
\item $\sigma \cdot A := P_\sigma A P_\sigma^{-1}$, where $P_\sigma \in \GL_n(\Z)$ is the permutation matrix of $\sigma \in S_n$.
\end{enumerate}
\end{enumerate}
\end{rem}

Keeping the above remark in mind, we can reformulate \cite[Proposition 1.5]{MR99} as follows:

\begin{prop}
The HW-manifolds $X$ and $X'$, with corresponding HW-matrices $A,A' \in \S^{n \times n}$, are affine equivalent if and only if $A$ and $A'$ are in the same orbit of the action of the group $G_n$.
\end{prop}

\section{Square distinguished matrices}

The following section is of a bit technical nature. Its purpose is to present some properties of square distinguished matrices. We start with a negative result:

\begin{lm}
\label{lm:not_existence}
Let $n>1$ be an integer. There does not exist a matrix $M \in \S^{n \times n}$ such that:
\begin{enumerate}[label={(A\arabic*)}]
\item \label{enum:distinguished+selfconjugate} $M$ is distinguished and self-conjugate;
\item the first row of $M$ is of the form $M_1 = [\s1,\s2,\ldots,\s2]$;
\item \label{enum:colsum_global} $\smc_i M = \s1$ for $1 \leq i \leq n$;
\item \label{enum:colsum_submatrices} in every principal submatrix of $M$ of odd degree there exists a column with sum of elements equal to $\s1$.
\end{enumerate}
\end{lm}

\begin{proof}
Assume that such a matrix $M$ exists. We will list some of its properties.
\begin{enumerate}[label={(P\arabic*)}]
\item \label{enum:permutations} Action by permutations of the set $\{2,3,\ldots,n\}$ on $M$, as in Remark \ref{rem:hwmats}, does not change its properties \ref{enum:distinguished+selfconjugate}--\ref{enum:colsum_submatrices}.
\item \label{enum:rowsum} 
$\smr_i(M) = \s1$ for every $1 \leq i \leq n$, since
\[
\smr_i(M) = \sum_{j=1}^n M_{ij} = \sum_{j=1}^n \conj{M_{ji}} = \conj{\sum_{j=1}^n M_{ji}} = \conj{\smc_i(M)} = \conj{\s1} = \s1.
\]
\item \label{enum:odd} $n$ is odd, by Lemma \ref{lm:column_sums}.
\item $M_{2,1} = \s3$ by self-conjugacy of $M$.
\item \label{enum:only2s} The second row of $M$ cannot be of the form $[\s3,\s1,\s2,\ldots,\s2]$, otherwise
\[
\smr_2(M) = \s3+\s1+(n-2)\s2 = \s2+\s2 = \s0,
\]
which contradicts \ref{enum:rowsum}.
\item \label{enum:only3s} The second row of $M$ cannot be of the form $[\s3,\s1,\s3,\ldots,\s3]$. Otherwise
\[
M = \begin{bmatrix}
* & A\\
* & B
\end{bmatrix},
\text{ where }
A = \begin{bmatrix}
\s2 & \ldots & \s2\\
\s3 & \ldots & \s3
\end{bmatrix} \in \S^{2 \times n-2}
\]
Using \ref{enum:colsum_global}, for every $i>2$ we get
\[
\s1 = \smc_i(M) = \s2+\s3+\smc_{i-2}(B) = \s1+\smc_{i-2}(B),
\]
hence $\smc_{i-2}(B)=\s0$ and this, together with \ref{enum:odd}, contradicts \ref{enum:colsum_submatrices}.
\item Using \ref{enum:permutations}, \ref{enum:only2s} and \ref{enum:only3s}, we can assume that 
\[
M_2 = [\s3,\s1,\underbrace{\s2,\ldots,\s2}_a,\underbrace{\s3,\ldots,\s3}_b],
\]
where $a,b>0$. Moreover, $a$ is even (and $b=n-2-a$ is odd), since
\begin{align*}
\s1 & = \smr_2(M) = \s3+\s1+a \cdot \s2 + b \cdot \s3 = \s2+a\cdot \s2+(n-2-a)\cdot \s3\\
  & = (1+a)\cdot \s2+(1+a)\cdot \s3 = (1+a)(\s2+\s3) = (1+a)\cdot \s1 = \s1+a\cdot \s1.
\end{align*}

\item \label{enum:only3s_submatrix} Let $M$ has the following block form
\[
M = \begin{bmatrix}
\s1 & \s2 & \s2 & \s2\\
\s3 & \s1 & \s2 & \s3\\
*   & *   & *   & C\\
*   & *   & *   & D\\
\end{bmatrix},
\]
where on the diagonal we have matrices of degree $1,1,a$ and $b$. There exists an element of $C$ equal to $\s2$. Otherwise, for every $i > a+2$, we have
\[
\s1 = \smc_i(M) = \s2+\s3+a \cdot \s3 + \smc_{i-a-2}(D) = \s1+\smc_{i-a-2}(D)
\]
and since $D$ is a principal submatrix of $M$ of odd degree, we get a contradiction with \ref{enum:colsum_submatrices}.
\end{enumerate}
By \ref{enum:only3s_submatrix} there exist $i$ and $j$, such that  $3 \leq i \leq a+2 < j \leq n$ and the principal submatrix $\Delta$ of $M$ given by indices $(2,i,j)$ is of the form
\[
\Delta = \begin{bmatrix}
\s1 & \s2 & \s3\\
* & \s1 & \s2\\
* & * & \s1
\end{bmatrix}.
\]
By self-conjugacy of $\Delta$ we immediately get
\[
\Delta = \begin{bmatrix}
\s1 & \s2 & \s3\\
\s3 & \s1 & \s2\\
\s2 & \s3 & \s1
\end{bmatrix},
\]
but this contradicts \ref{enum:colsum_submatrices}.
\end{proof}

\begin{rem}
To a logical sentence $\Theta$ we assign (in a natural way) an element $[\Theta] \in \Z_2$ as follows:
\[
[\Theta] = 1 \Leftrightarrow \Theta \text{ is true}.
\]
\end{rem}

\begin{rem}
Let $n \in \N, M \in \S^{n \times n}$ and $U \in \PA_n$. By $M_U$ we denote the sum of the rows of $M$ from the set $U$:
\[
M_U := \sum_{i \in U} M_i
\]
and $M_U[j]$ -- its $j$-th coordinate, for $1 \leq j \leq n$. We get 
\[
J_M(U) = \left\{ j : \smc_j^U(M) = \s1 \right\} = \{ j : M_U[j]=\s1 \}.
\]
\end{rem}

The following lemma, which describes map $J$ for distinguished matrices, extends \cite[Proposition 3]{PS16}.

\begin{lm}
\label{lm:jmap}
Let $n \in \N$, $M \in \S^{n \times n}$ be distinguished and $S,U \in \PA_n$. The following hold:
\begin{enumerate}
\item \label{e:dist:1} $J_M(U) = U$ if $|U|=1$.
\item \label{e:dist:odd} $J_M(U) \subset U$ if $|U|_2=1$.
\item \label{e:dist:even} $J_M(U) \cdot U = 0$ if $|U|_2=0$.
\item \label{e:dist:odd1} $|J_M(U)|_2 = \sum_{i,j \in U} M_{ij}$ if $|U|_2=1$.
\item \label{e:dist:even1} $|J_M(U)|_2 = \sum_{i,j \in U} M_{ij}+\sum_{i \in U} \smr_i(M)$ if $|U|_2=0$.
\item \label{e:dist:odd2} $|J_M(U)S|_2 = \sum_{j \in U}[j \in S]M_U[j]$ if $|U|_2=1$.
\item \label{e:dist:even2} $|J_M(U)S|_2 = \sum_{j \in U}[j \in S]M_U[j] + \sum_{i \in U}\smr_i^S(M)$ if $|U|_2=0$.
\end{enumerate}
\end{lm}


\begin{proof}
Property \ref{e:dist:1} holds just because $M$ is distinguished -- in fact, we have 
\begin{equation}
\label{eq:map_j_on_sigletons}
\forall_{1 \leq i \leq n} J_M(\{i\}) = \{i\}.
\end{equation}
Properties \ref{e:dist:odd} and \ref{e:dist:even} hold by the same rule as in the proof of Lemma \ref{lm:column_sums}. This rule will be also used in the rest of the proof.

Note that \ref{e:dist:odd1} and \ref{e:dist:even1} follow from \ref{e:dist:odd2} and \ref{e:dist:even2} respectively, if one takes $S=\{1,\ldots,n\} = 1 \in \PA_n$.

Recall Remark \ref{rem:injection}, by which $\Z_2$ is a subgroup of $\S$.

If $|U|$ is odd then $M_U[j] \in \{\s0,\s1\}$ if and only if $j \in U$ and $[j \in J_M(U)] = M_U[j] \cdot [j \in U]$ for $1 \leq j \leq n$, hence
\[
|J_M(U)S|_2 = \sum_{j=1}^n [j \in S][j \in J_M(U)] = \sum_{j=1}^n [j \in S][j \in U]M_U[j] = \sum_{j \in U}[j \in S]M_U[j].
\]

If $|U|$ is even on the other hand, we get that $M_U[j] \in \{\s0,\s1\}$ if and only if $j \not\in U$ and $[j \in J_M(U)] = M_U[j] \cdot [j \not\in U]$. In a similar fashion as above we have
\begin{align*}
|J_M(U)S|_2 &= \sum_{j=1}^n [j \in S][j \in J_M(U)] = \sum_{j=1}^n [j \in S][j \not \in U]M_U[j]\\
&= \sum_{j \in U} [j \in S]M_U[j] +\sum_{j=1}^n [j \in S]M_U[j]\\
&= \sum_{j \in U} [j \in S]M_U[j] +\sum_{j=1}^n [j \in S]\sum_{i \in U}M_{ij}\\
&= \sum_{j \in U} [j \in S]M_U[j] +\sum_{i \in U} \sum_{j \in S} M_{ij} = \sum_{j \in U} [j \in S]M_U[j] + \sum_{i \in U} \smr_i^S(M).
\end{align*}
\end{proof}

Directly from the definition of HW-matrices and the above lemma we get:

\begin{cor}[{\cite[Proposition 3]{PS16}}]
\label{cor:jmap_for_hw}
Let $M$ be a HW-matrix. Then:
\begin{enumroman}
\item $J_M(1) = 0$;
\item $J_M(U) \neq 0$ for $U \in \PA_n \setminus \{0,1\}$;
\item $J_M(U) = J_M(1+U)$.
\end{enumroman}
\end{cor}

\section{\Spinc structures and HW-matrices}

In this section we give a necessary and sufficient condition for existence of a \spinc structure on a manifold defined by a HW-matrix. 
Let us note an easy lemma.

\begin{lm}
Let $d\in \N$. A map $\kappa_A \colon \Z_2[x_1,\ldots,x_d] \to \Map(\PA_d,\Z_2)$ defined by
\[
\kappa_A(x_i)(U) = [ i \in U ],
\]
where $1 \leq i \leq d$ and $U \in \PA_d$, is an algebra homomorphism.
%
\end{lm}


We will use the following properties of the map $\kappa_A$:

\begin{lm}
\label{lm:kappa}
Let $d,n \in \N$ and $A \in \S^{d \times n}$. Then:
\begin{enumroman}
\item $\kappa_A$ is a monomorphism in gradation $2$;
\item $\kappa_A(\theta^A_j)(U) = [j \in J_A(U)]$.
\end{enumroman}
\end{lm}

\begin{proof}
Let $\kappa = \kappa_A$ and
\[
x = \sum_{1 \leq i<j \leq d} \alpha_{ij} x_ix_j \in \ker\kappa,
\]
where $\alpha_{ij} \in \Z_2$. For any $1 \leq k<l \leq d$ and $U=\{k,l\}$ we have
\[
0 = \kappa\left(\sum_{1 \leq i<j \leq d} \alpha_{ij} x_ix_j\right)(U) = \sum_{1 \leq i<j \leq d} \alpha_{ij} \kappa(x_i)(U) \cdot \kappa(x_j)(U) = \alpha_{kl},
\]
hence $x = 0$.

Now take $1 \leq j \leq n$. We have
\[
\theta_j = \alpha_j \beta_j = \left( \sum_{i=1}^d \alpha(A_{ij})x_i \right) \left( \sum_{k=1}^d \beta(A_{kj})x_k \right)
\]
and in the consequence, for any $U \in \PA_d$,
\[
\kappa(\theta_j)(U) = \left( \sum_{i \in U} \alpha(A_{ij}) \right) \left( \sum_{k \in U} \beta(A_{kj}) \right).
\]
Denote by $a,b,c,d$ the number of $\s0,\s1,\s2,\s3$ in the rows from the set $U$ of $j$-th column of $A$, respectively. We get $\kappa(\theta_j)(U) = (b+c)(b+d) \bmod 2$, but
\[
(b+c)(b+d) \bmod 2 = 1 \Leftrightarrow (b+c) \bmod 2 = (b+d) \bmod 2  = 1.
\]
Hence $\kappa(\theta_j)(U)=1$ if and only if
\begin{align*}
\s1 &= (b+c)\cdot \s2 + (b+d) \cdot \s3\\
    &= b \cdot (\s2+\s3) + c \cdot \s2 + d \cdot \s3\\
    &= a \cdot \s0 + b \cdot \s1 + c \cdot \s2 + d \cdot \s3 = \smc_j^U(A),
\end{align*}
which by definition means, that $j \in J_A(U)$.
\end{proof}

\begin{prop}
\label{prop:squares}
Let $n>1$ be an odd integer and let $A \in \S^{n-1 \times n}$ be distinguished. The following conditions are equivalent:
\begin{enumerate}
\item \label{prop:sq:1} There exists $x \in H^1(C_2^{n-1},\Z_2)$ such that $x^2 + \sw^A_2 \in \vspan\{\theta^A_1,\ldots,\theta^A_n\}$.
\item \label{prop:sq:2} $\sigma_2 \in V_\delta:=\vspan\{ \theta^A_1-x_1^2,\ldots,\theta^A_{n-1}-x_{n-1}^2, \theta^A_n \}$, where $\sigma_2$ is the elementary symmetric polynomial of degree $2$ in variables $x_1,\ldots,x_n$.
\item \label{prop:sq:3} There exists $S \in \PA_n$, such that for every $U \in \PA_{n-1}$ the equality (in $\Z_2$) holds
\begin{equation}
\label{eq:cond_almost_spin}
|(J_A(U)+U)S|_2 = \binom{|U|}{2}.
\end{equation}
\end{enumerate}
\end{prop}

\begin{proof}
We will omit the super and subscript $A$ in the proof.

Denote by $V$ the subspace of $\Z_2[x_1,\ldots,x_{n-1}]$ of polynomials of degree $2$. Let $V_s$ and $V_f$ be subspaces ov $V$ generated by monomials which are and are not squares, respectively. Let $p \colon V \to V_f$ be the projection coming from the decomposition $V = V_s \oplus V_f$. 
Note that 
\[
p(\theta_j) = \theta_j - x_j^2 \text{ and } p(\theta_n) = \theta_n
\]
for $1 \leq j < n$, hence condition \ref{prop:sq:1} is equivalent to
\begin{equation}
\label{eq:pif}
p(\sw_2) \in \vspan\{p(\theta^A_1),\ldots,p(\theta^A_n)\} = V_\delta,
\end{equation}
but directly from the formula \eqref{eq:sw}, since $n$ is odd, we have that $p(\sw_2) = \sigma_2$.

Assume $1 \leq j \leq n$ and let $\delta_j := p(\theta_j)$. For $U \in \PA_{n-1}$ we have that
\begin{equation}
\label{eq:kappa_delta}
\kappa(\delta_j)(U) = [ j \in J(U)+U ].
\end{equation}
Indeed, if $j < n$, using Lemma \ref{lm:kappa} we get
\begin{align*}
\kappa(\delta_j)(U) & = \kappa(\theta_j+x_j^2)(U) = \kappa(\theta_j)(U)+\kappa(x_j^2)(U) \\
& = \kappa(\theta_j)(U)+\kappa(x_j)(U)^2 = \kappa(\theta_j)(U)+\kappa(x_j)(U)\\
& = [j \in J(U)] + [j \in U] = [ j \in J(U)+U ]. 
\end{align*}
Additionally, $\delta_n = \theta_n$ and $n \not\in U$, hence
\[
\kappa(\delta_n)(U) = [n \in J(U)] = [n \in J(U)] + [n \in U] = [ n \in J(U)+U ].
\]

Suppose that $\sigma_2 = \sum s_j \delta_j \in V_\delta$ and let $S := \{ j : s_j = 1 \} \in \PA_n$. For every $U \in \PA_{n-1}$ we have
\[
\kappa(\sigma_2)(U) = \sum_{j=1}^n s_j \kappa(\delta_j)(U).
\]
Since
\[
\kappa(\sigma_2)(U) = \sum_{1 \leq k < l < n} [l \in U][k \in U] = \mathop{\sum_{k,l \in U}}_{k < l} 1 = \binom{|U|}{2}
\]
and
\begin{align*}
\sum_{j=1}^n s_j \kappa(\delta_j)(U) & = \sum_{j=1}^n [j \in S][j \in J(U)+U] \\ &= \sum_{j=1}^n [j \in S\cdot(J(U)+U)] = |S\cdot(J(U)+U)|_2,
\end{align*}
formula \eqref{eq:cond_almost_spin} follows.

Now assume that \eqref{eq:cond_almost_spin} holds for some $S \in \PA_n$ and every $U \in \PA_{n-1}$. By the above calculations it may me written as
\[
\sum_{j=1}^n [j \in S][j \in J(U)+U] = \kappa(\sigma_2)(U).
\]
Put $s_j = [j \in S]$ and use \eqref{eq:kappa_delta}. The above equation takes the form
\[
\sum_{j=1}^n s_j \kappa(\delta_j)(U) = \kappa(\sigma_2)(U).
\]
Recall that $U$ is any element of $\PA_{n-1}$. Using this and the linearity of $\kappa$, we get
\[
\kappa\left( \sum s_j \delta_j \right) = \kappa(\sigma_2).
\]
By Lemma \ref{lm:kappa}, $\sigma_2 = \sum s_j \delta_j \in V_\delta$.
\end{proof}

\begin{df}
Let $n \in \N, M \in \S^{n \times n}$ and $S \in \PA_n$.
\begin{enumerate}
\item We call $S$ a \emph{\spinc set} for $M$ if for every $U \in \PA_n$ the equation
\begin{equation}
\label{eq:cond_spin}
|(J_M(U)+U)S|_2 = \binom{|U|}{2}
\end{equation}
holds;
\item We call $S$ an \emph{almost \spinc set} for $M$ if for every $U \in \PA_{n-1}$ equation \eqref{eq:cond_spin} holds.
\end{enumerate}
If $S$ is a \spinc set for $M$, we call $(M,S)$ a \emph{\spinc pair}.
\end{df}

\begin{lm}
\label{lm:almost_spinc_set}
Let $n \in \N$ be odd, $M \in \HWM_n$ and $S \in \PA_n$.
\begin{enumerate}
\item If $S$ is an almost \spinc set for $M$, then 
\[
|S|_2 = \frac{n-1}{2}.
\]
\item If $S$ is an almost \spinc set set for $M$, then it is a \spinc set for $M$.
\end{enumerate}
\end{lm}

\begin{proof}
Take $U = \{1,\ldots,n-1\}$. By Lemma \ref{lm:jmap} and Corollary \ref{cor:jmap_for_hw}, $J(U) = J(1+U) = J(\{n\}) = \{n\}$. Hence $J(U)+U = \{1,\ldots,n\}=1$, $(J(U)+U)S=S$ and we get
\[
|S|_2 = |(J(U)+U)S|_2 = \binom{|U|}{2} = \binom{n-1}{2} = \frac{n-1}{2}.
\]
Note again, that all equations above are in $\Z_2$. In particular the last one holds, because $n$ is odd.

Assume now that $S$ is an almost \spinc set for $M$. Equation \eqref{eq:cond_spin} holds for every $U \in \PA_{n-1}$. It is enough to show that it also holds whenever $n \in U$. In that case however $V = 1+U \in \PA_{n-1}$, so we have
\[
|(J(V)+V)S|_2 = \binom{|V|}{2}.
\]
By Corollary \ref{cor:jmap_for_hw}, $J(V) = J(1+U) = J(U)$, hence
\[
(J(U)+U)S = (J(V)+V+1)S = (J(V)+V)S + S
\]
and by linearity of $|\cdot|_2$ we have
\begin{align*}
|(J(U)+U)S|_2 = |(J(V)+V)S|_2 + |S|_2 & = \binom{|V|}{2}+\frac{n-1}{2}\\
& = \binom{n-|U|}{2} + \frac{n-1}{2} = \binom{|U|}{2},
\end{align*}
where in the last equality we again use the fact, that $n$ is odd.
\end{proof}

\begin{thm}
\label{thm:spinc_condition2}
Let $n \in \N, n \geq 5, M \in \HWM_n$ and let $X$ be the HW-manifold defined by $M$.The following conditions are equivalent:
\begin{enumerate}
\item $X$ admits a \spinc structure.
\item There exists a \spinc set for $M$.
\end{enumerate}
\end{thm}
\begin{proof}
By Lemma \ref{lm:almost_spinc_set} existence of a \spinc and an almost \spinc set are equivalent conditions.
Let $A$ be a matrix composed from the first $n-1$ rows of $M$. Clearly it is distinguished and by Remark \ref{rem:hwmats}, $A$ is defining and effective matrix for $X$. In order to get the desired equivalence, notice that for every $U \in \PA_{n-1}$ the equality
\[J_A(U) = J_M(U)\] 
holds, use Theorem \ref{thm:spinc_condition1} and Proposition \ref{prop:squares}.
\end{proof}


\section{Standard forms of \spinc pairs}

Recall that in Remark \ref{rem:hwmats} we have defined the action of the group $G_n = C_2 \wr S_n$ on the space $\S^{n \times n}$, for every $n \in \N$. We will show that in fact it can act on \spinc pairs.

\begin{lm}
\label{lm:equivalent_spinc_pairs}
Let $n \in \N, M \in \S^{n \times n}, S \in \PA_n$ be such that $(M,S)$ is a \spinc pair. Then for every $\sigma \in S_n$, $(\sigma M, \sigma S)$ is also a \spinc pair.
\end{lm}

\begin{proof}
Let $U \in \PA_n$ and $\sigma \in S_n$. Using an easy observation that $J_{\sigma M}(U) = \sigma J_M(\sigma^{-1}U)$ and Lemma \ref{lm:set_algebra}, we get
\begin{align*}
\left|\bigl(J_{\sigma M}(U)+U\bigr)(\sigma S)\right|_2 & = 
\left|\bigl(\sigma J_M(\sigma^{-1}U)+U\bigr)(\sigma S)\right|_2 =\\ 
&= \left| \sigma \biggl( \bigl( J_M(\sigma^{-1}U)+\sigma^{-1}U\bigr) S\biggr) \right|_2\\
& = \left| \bigl( J_M(\sigma^{-1}(U))+\sigma^{-1}(U) \bigr)S \right|_2 = \binom{|\sigma^{-1}(U)|}{2} = \binom{|U|}{2}.
\end{align*}
\end{proof}

Note, with the assumptions of the above lemma, that $G_n$ acts on $\PA_n$ by permutations, using the canonical epimorphism $G_n \to S_n$. 
Moreover, if $g \in G_n$ is an element which acts by conjugations of columns only, then $J_{gM} = J_M$, since $\conj{\s1} = \s1$. We immediately get

\begin{cor}
\label{cor:equivalent_spinc_pairs}
Let $n \in \N, M \in \S^{n \times n}, S \in \PA_n$ be such that $(M,S)$ is a \spinc pair. Then for every $g \in G_n$, $(gM, gS)$ is also a \spinc pair.
\end{cor}

\begin{lm}
\label{lm:standard_form_1}
Let $n \in \N$ and $M \in \S^{n \times n}$ be distinguished and such that
\[
|J_M(U)|_2 = 1
\]
for every two-element set $U \in \PA_n$. Then there exists an integer $k$, such that $2k \geq n$ and in the orbit $G_nM$ there exists a matrix $M'$ in the following block form
\[
M' = \begin{bmatrix}
A & C\\
C^t & B
\end{bmatrix},
\]
where $A$ and $B$ are self-conjugate of degree $k$ and $n-k$, respectively. Moreover
\begin{equation}
\label{eq:sum_rows_groups}
\smr_1(M') = \ldots = \smr_k(M') \neq \smr_{k+1}(M') = \ldots = \smr_n(M').
\end{equation}
\end{lm}

\begin{proof}
Since the matrix $M^t$ is distinguished, by Lemma \ref{lm:column_sums} we get that the set $\{ \smr_i(M) : 1 \leq i \leq n \}$ has at most two elements. Let $l = |\{ i : \smr_i(M) = \smr_1(M) \}|$. If $2l \geq n$ take $k=l$ and $M''=M$. Otherwise, construct $M''$ by conjugation of the first column of $M$. We have $\smr_1(M'') = \smr_1(M)$ and $\smr_i(M'') = \smr_i(M)+\s1$ for $i > 1$. Letting $k = |\{ i : \smr_i(M'') = \smr_1(M'') \}|$ we have $2k \geq n$.

There exists a permutation $\sigma \in S_n$, which fixes $1$ and such that $M' = \sigma M''$ is of the block form
\[
\begin{bmatrix}
A & C\\
D & B
\end{bmatrix},
\]
where $A,B$ are of degrees $k,n-k$ respectively and the equation \eqref{eq:sum_rows_groups} holds.

Let $U = \{i,j\}$ for $1 \leq i < j \leq n$. By our assumptions and Lemma \ref{lm:jmap} we have
\[
\s1 = M'_{ii} + M'_{ij} + M'_{ji} + M'_{jj} + \smr_i(M') + \smr_j(M')
\]
and hence
\begin{equation}
\label{eq:two_element_sum}
M'_{ij}+M'_{ji} = \smr_i(M') + \smr_j(M') + \s1
\end{equation}
Consider two cases:
\begin{enumerate}
\item $j \leq k$ or $i > k$. Equation \eqref{eq:two_element_sum} gives us $M'_{ij}+M'_{ji}=\s1$ and since $M'$ is distinguished, $M'_{ij} = \conj{M'_{ji}}$. Hence $A$ and $B$ are self-conjugate.
\item $i \leq k < j$ and hence $\smr_i(M') = \smr_j(M')+\s1$. Equation \eqref{eq:two_element_sum} gives us $M'_{ij}=M'_{ji}$, hence $D = C^t$. 
\end{enumerate}
\end{proof}

\begin{lm}
\label{lm:standard_form_2}
Let $n \in \N$ and $M \in \S^{n \times n}$ be distinguished in the following block form
\[
M = \begin{bmatrix}
A & C\\
C^t & B
\end{bmatrix},
\]
where $A,B$ are of degrees $k,l$, respectively. Assume that $k > 0$ and:
\begin{enumroman}
\item $A$ is self-conjugate;
\item $M_1 = [\s1,\s2,\ldots,\s2]$;
\item $J_M(\{1,i,j\}) \neq 0$ for $1 \leq i \leq k < j \leq n$.
\end{enumroman}
Then $C$ consists only of elements equal to $\s2$.
\end{lm}

\begin{proof}
If $l=0$, there is nothing to prove. Assume that $l>0$, take $i \leq k$ and $j > k$. The principal submatrix of $M$ defined by indices $1,i,j$ is of the form:
\[
\begin{bmatrix}
\s1 & \s2 & \s2\\
\s3 & \s1 & x\\
\s2 & x & \s1
\end{bmatrix}
\]
If $x=\s3$ then $J_M(\{1,i,j\})=0$, contrary to our assumptions, hence $M_{ij} = x = \s2$. Together with the form of $M_1$, we get the desired result.
\end{proof}

\begin{df}
Let $n \in \N, M \in \S^{n \times n}$ and $S$ be a \spinc set for $M$. We will say that the \spinc pair $(M,S)$ is in \emph{standard form} if:
\begin{enumroman}
\item $S = \{ 1,\ldots,|S| \}$;
\item $M_1 = [\s1,\s2,\ldots,\s2]$;
\item $M$ is distinguished and in the block form
\[
\begin{bmatrix}
A & \s2 & *\\
\s2 & B & *\\
* & * & *\\
\end{bmatrix}
\]
with elements on the diagonal of degrees $k,l,r$;
\item $k \geq l$ and $k+l = |S|$;
\item $A,B$ are self-conjugate;
\item $\smr_1^S(M) = \smr_k^S(M) \neq \smr_{k+1}^S(M) = \ldots = \smr_{k+l}^S(M)$ (it is possible that $l=0$).
\end{enumroman}
\end{df}

We can deduce some further restrictions on a standard form of a matrix.

\begin{lm}
\label{lm:form_outside_support}
Keeping the notation from the above definition, 
let $(M,S)$ be a \spinc pair in the standard form and $k+l < m \leq n$. Then, in the block form
\[
M_m = 
\begin{bmatrix}
a & \conj{a} & *
\end{bmatrix},
\] 
where $a \in \{\s2,\s3\}$ is such that the equation
\begin{equation}
\label{eq:a}
ka+l\conj{a} + (k-l-1)\s2 = a
\end{equation}
holds.
\end{lm}

\begin{proof}
Let $i \leq k+l$. Using the fact that $(M,S)$ is a \spinc pair in the standard form and Lemma \ref{lm:jmap}, for $U=\{i,m\} $ we get
\begin{align*}
1 = \binom{|U|}{2} &= |(J_M(U)+U)S|_2 = |J_M(U)S|_2+|US|_2\\
&= [i \in S](M_{ii}+M_{mi}) + [m \in S](M_{im}+M_{mm}) \\
&\phantom{=} + \smr^S_i(M) + \smr^S_m(M) + [i \in S] + [m \in S]\\
&= M_{mi} + \smr^S_i(M) + \smr^S_m(M)
\end{align*}
and hence 
\[
M_{mi} = \smr^S_i(M) + \smr^S_m(M) + \s1 = \left\{
\begin{array}{ll}
\smr^S_1(M) + \smr^S_m(M) + \s1 & \text{ if } i \leq k\\[.5em]
\smr^S_1(M) + \smr^S_m(M) & \text{ if } i >k\\
\end{array}
\right.
\]
Since $M$ is distinguished and $i < m$, setting $a:=M_{m1}$ gives us desired form of the $m$-th row of $M$. The equation \eqref{eq:a} follows from the fact that $\smr_1^S = \s1+(k+l-1)\s2$ and $\smr_m^S = ka+l\conj{a}$.
\end{proof}

By the following lemma, certain \spinc pairs can be transformed to standard forms.

\begin{lm}
\label{lm:standard_form_3}
Let $n \geq 3$ be an odd integer and $M \in \S^{n \times n}$ be distinguished. Let $S$ be a \spinc set for $M$. If
\[
J_M(U) \neq 0 \text{ for } U \subset S \text{ and } |U| = 3,
\]
then there exists $g \in G_n$ such that $(gM, gS)$ is a \spinc pair in a standard form.
\end{lm}

\begin{proof}
By Corollary \ref{cor:equivalent_spinc_pairs} $(gM,gS)$ is a \spinc pair for any $g \in G_n$. Our goal is to show that $(M,S)$ can be transformed to a pair in the standard form.

By permuting indices and conjugating columns, we can transform $(M,S)$ to a form where $S=\{1,\ldots,|S|\}$ and $M_1=[\s1,\s2,\ldots,\s2]$.

Let $N$ be the principal submatrix of $M$ defined on the set $S$. $N$ is distinguished and for every $U \in \PA(S) = \PA_{|S|}$ we have
\[
|J_N(U)+U|_2 = |(J_N(U)+U)S|_2 = |(J_M(U)+U)S|_2 = \binom{|U|}{2}.
\]
In particular, $|J_N(U)|_2=1$ if $|U|=2$. Using Lemma \ref{lm:standard_form_1} we can act on $M$ by an element of $G_{|S|} \subset G_n$ such that $N$ becomes
\[
N = \begin{bmatrix}
A & C \\
C^t & B
\end{bmatrix},
\]
where $A$ and $B$ are self-conjugate of degrees $k,l$ respectively, such that $k \geq l$ and
\[
\smr_1(N) = \ldots = \smr_k(N) \neq \smr_{k+1}(N) = \ldots = \smr_{k+l}(N).
\]
Note that $\smr_i(N) = \smr_i^S(M)$ for $1 \leq i \leq |S|=k+l$.

By assumption and Lemma \ref{lm:jmap} we have
\[
J_N(U) = J_M(U)S = J_M(U) \neq 0
\]
for $U \subset S$ and $|U|=3$. By Lemma \ref{lm:standard_form_2} we get that $C=\s2$ and hence the \spinc pair $(M,S)$ was transformed to a standard form.
\end{proof}

\section{\Spinc structures on HW-manifolds}


By the results of previous sections we know that the existence of a \spinc structure on a HW-manifold is equivalent to the existence of a \spinc set for its HW-matrix. We will show that this never happens in dimensions greater than $3$.

\begin{lm}
\label{lm:spinc_pair_n}
Let $n \geq 5$ be an odd integer and $M \in \HWM_n$. There does not exist a \spinc set $S$ for $M$ such that $|S|=n$.
\end{lm}

\begin{proof}
If such a set $S$ exists, then by our assumptions $J_M(U) \neq 0$ for $|U|=3$ and by Lemma \ref{lm:standard_form_3} we can assume that $(M,S)$ is in a standard form:
\[
M = \begin{bmatrix}
A & \s2\\
\s2 & B
\end{bmatrix},
\]
where the degrees of $A,B$ equal $k,l$ respectively, $k \geq l$ and:
\[
\smr_i(M) = \left\{
\begin{array}{ll}
\s1 & \text{ if } i \leq k\\
\s0 & \text{ if } i > k
\end{array}
\right.
\]
By definition of HW-matrices we have
\[
\s0 = \sum_{j=1}^n \smc_j(M) = \sum_{i=1}^n \smr_i(M) = k \cdot \s1,
\]
hence $k$ is even and in particular $k<n$.

Let $U=\{1,\ldots,k\}$. Since it is of even size, $J_M(U)U=0$ by Lemma \ref{lm:jmap}. Moreover, for every $j > k$ we have
\[
M_U[j] = \sum_{i \in U} M_{ij} = \sum_{i \in U} \s2 = \s0.
\]
Hence $J_M(U) = 0$. Contradiction with the fact that $M \in \HWM_n$.
\end{proof}

\begin{lm}
\label{lm:spinc_pair_n-1}
Let $n \geq 5$ be an odd integer and $M \in \HWM_n$. There does not exist a \spinc set $S$ for $M$ such that $|S|=n-1$.
\end{lm}

\begin{proof}
Similarly as in the proof of the previous lemma we can assume that
\[
M = \begin{bmatrix}
A & \s2 & *\\
\s2 & B & *\\
* & * & \s1
\end{bmatrix}
\]
where the matrices on the diagonal are of degrees $k,l,1$ respectively, $k \geq l$ and $M_1 = [\s1,\s2,\ldots,\s2]$.

Since $k+l=n-1$ is even, $k=l \mod 2$. By Lemma \ref{lm:form_outside_support} we get
\[
M_n = [a,\conj{a},\s1] \text{ and } k\cdot\s1 + \s2 = a.
\]

If $k$ is odd, then $l$ is odd and $a=\s3$. By definition of a HW-matrix, we get $\smc_i(B)=\s1$ for some $k+1\leq i \leq k+l$ and
\[
\s0 = \smc_{k+i}(M) = k \cdot \s2 + \smc_i(B) + \s2 = \smc_i(B) = \s1,
\]
a contradiction.

Assume that $k$ is even. Then $l$ is even and $a=\s2$. If $l=0$, then $M_n=[\s2,\ldots,\s2,\s1]$ and $J_M(\{1,n\})=0$, which cannot happen. Suppose $l>0$. Take $U=\{1,\ldots,k\}, V=\{k+1,\ldots,l\}$. They are both sets of even size. By the form of $M$ and Lemma \ref{lm:column_sums} we have
\[
M_U[i] \in \{\s2,\s3\} \text{ and } M_V[i] = l \cdot \s2 = 0 \text{ if } i \leq k
\]
and
\[
M_U[i] = k \cdot \s2 = 0 \text{ and } M_V[i] \in \{\s2,\s3\} \text{ if } k < i < n.
\]
Since $M$ is a HW-matrix, we get $M_U[n] = M_V[n] = \s1$, but then
\[
\s0 = \smc_n(M) = M_U[n]+M_V[n]+\s1 = \s1,
\]
a contradiction.
\end{proof}

\begin{lm}
\label{lm:spinc_pair_n-2}
Let $n \geq 5$ be an odd integer and $M \in \HWM_n$. There does not exist a \spinc set $S$ for $M$ such that $|S|=n-2$.
\end{lm}

\begin{proof}
Similarly as in the previous two cases, we may assume that
\[
M = \begin{bmatrix}
A & \s2 & * & *\\
\s2 & B & * & *\\
* & * & \s1 & *\\
* & * & * & \s1
\end{bmatrix},
\]
where the blocks on the diagonal are of degrees $k,l,1,1$, respectively and $k \geq l$. Since $k+l = n-2$ is odd, $k=l+1 \mod 2$. By Lemma \ref{lm:form_outside_support} we have
\[
M_{n-1} = [a,\conj{a}, \s1, *] \text{ and } k\cdot \s1 + \conj{a} = a,
\]
hence $k \cdot \s1 = \s1$, $k$ is odd and $l$ is even.

Assume that $M_n = [b,\conj{b}, *, \s1]$. We have $a \neq b$, otherwise
\[
M_{n-1}+M_n = [\s0, \ldots, \s0, c,d],
\]
where $c,d \in \{\s2,\s3\}$, hence $J_M(\{n-1,n\})=0$.


For every $i \leq k$ we get
\[
\s0 = \smc_i(M) = \smc_i(A) + l \cdot \s2 + \s2 + \s3 = \smc_i(A) + \s1,
\]
hence $\smc_i(A) = \s1$. But by Lemma \ref{lm:not_existence} matrix $A$ cannot exist, a contradiction.
\end{proof}

\begin{prop}
\label{prop:spinc_pair}
Let $n\geq 5$ be an odd integer and $M \in \HWM_n$. There does not exist a \spinc set for $M$.
\end{prop}

\begin{proof}
Let $S$ be a \spinc set for $M$. By Lemmas \ref{lm:spinc_pair_n}, \ref{lm:spinc_pair_n-1} and \ref{lm:spinc_pair_n-2} we can assume that $|S|\leq n-3$. In this case there exists a set $U \in \PA_n$ of size $3$ such that $US=0$. By Lemma \ref{lm:jmap} $J_M(U) \subset U$, hence $(J_M(U)+U)S=0$. Since $S$ is a \spinc set for $M$, we have
\[
0 = |(J_M(U)+U)S|_2 = \binom{|U|}{2} = \binom{3}{2} = 1,
\]
a contradiction.
\end{proof}

Finally we are ready to state the main result of the paper:
\begin{thm}
Let $X$ be a Hantzsche-Wendt manifold of dimension $n \geq 5$. Then $X$ does not admit a \spinc-structure.
\end{thm}

\begin{proof}
This follows directly from Theorem \ref{thm:spinc_condition2} and Proposition \ref{prop:spinc_pair}.
\end{proof}

\printbibliography

\end{document}